\documentclass[12pt]{amsart}
\usepackage[left=1in,right=1in,top=1in,bottom=1in]{geometry} 
\usepackage{url}
\usepackage{amssymb,amsmath,color}
\usepackage{graphicx}
\usepackage{enumerate}
\usepackage{pgfplots}
\pgfplotsset{compat=1.15}
\usepackage{mathrsfs}
\usetikzlibrary{arrows}
\DeclareGraphicsExtensions{.eps}
\newtheorem{theorem}{Theorem}[section]

\newtheorem{corollary}[theorem]{Corollary}
\theoremstyle{definition}

\newtheorem{example}[theorem]{Example}

\newtheorem{remark}[theorem]{Remark}
\numberwithin{equation}{section}
\title[$\mathcal{AN}$ quasi-$\ast$-paranormal operators]{Some Classes of Absolutely Norm Attaining Weighted Shift operators on Directed Graphs}
\author{K KRISHNAN}
\address{K Krishnan\newline Department of Mathematics, University College (Affiliated to University of Kerala), Thiruvananthapuram, Kerala,
	India- 695034.}
\email{krishnank@universitycollege.ac.in}
\author{T. PRASAD}
\address{T. Prasad \newline Department of Mathematics, University of Calicut, Kerala-673635, India.}
\email{prasadvalapil@gmail.com}
\author{E. SHINE LAL}
\address{E. Shine Lal \newline Department of Mathematics, University College (Affiliated to University of Kerala), Thiruvananthapuram, Kerala,
	India- 695034.}
\email{shinelal@universitycollege.ac.in}
\keywords{Norm attaining operators, absolutely norm attaining operators, quasi-$\ast$-paranormal operators, weighted shifts on directed trees.}
\subjclass[2010]{47A10, 47A15, 47A65, 47A67, 47B20}
\begin{document}
	\maketitle
	\begin{abstract}
		In this paper we study absolutely norm attaining quasi-$\ast$-paranormal  weighted shifts on directed graphs and give some examples. Moreover we give some examples which show that the spectrum of a positive absolutely norm attaining operator containing more than one eigenvalue with infinite multiplicity. Later we investigate weighted composition and Lambert operators on directed graphs.
	\end{abstract}
	\section{Introduction and Preliminaries}
	Let $\mathcal{B}(\mathcal{H}_1,\mathcal{H}_2)$ denote the class of all bounded linear operators between Hilbert spaces $\mathcal{H}_1$ and $\mathcal{H}_2$ and $\mathcal{B(H)}$, the space of all operators defined on a Hilbert space $\mathcal{H}$. We denote the class of all norm attaining operators on a Hilbert space $\mathcal{H}$ by $\mathcal{N(H)}$ and the class of all absolutely norm attaining operators on $\mathcal{H}$ by $\mathcal{AN(H)}$. The class of norm attaining and absolutely norm attaining operators were introduced by Carvajal and Neves \cite{Operatorsthatachievethenorm}. An operator $T \in \mathcal{B}(\mathcal{H}_1,\mathcal{H}_2)$ is said to be norm attaining if there exists a non-zero $x \in \mathcal{H}_1$ such that $\|Tx\|=\|T\|\|x\|$ and is said to be absolutely norm attaining if $T|_M$ is norm attaining for every closed subspace $M$ of $\mathcal{H}_1$.  It is clear that $\mathcal{AN(H)} \subset \mathcal{N(H)}$. For more details, see \cite{Operatorsthatachievethenorm}.
	
	We denote the range space and null space of the operator $T \in \mathcal{B(H)}$ by $R(T)$ and $N(T)$ respectively. An operator $T$ is said to be an isometry if $\|Tx\|=\|x\|$ for every $x \in \mathcal{H}$ and is said to be a partial isometry if $T|_{N(T)^{\perp}}$ is an isometry where $N(T)^{\perp}$ is the orthogonal complement of $N(T)$ \cite{furuta}.  An operator $T \in \mathcal{B}(\mathcal{H})$ is normal if $T^*T=TT^*$, hyponormal if $\|T^*x\| \leq \|Tx\|$, paranormal if $\|Tx\|^2 \leq \|T^2x\|\|x\|$, $\ast$-paranormal if $\|T^*x\|^2 \leq \|T^2x\|\|x\|$ and quasi-$\ast$-paranormal operator if 
	$\|T^*Tx\|^2 \leq \|T^3x\|\|Tx\| \text{ for all } x \in \mathcal{H}$ where $T^*$ denotes the adjoint of $T$ \cite{furuta, quasi*}. In general, the following relation holds.
	\[
	\text{hyponormal $\subseteq$ $\ast$-paranormal $\subseteq$ quasi-$\ast$-paranormal}
	\]
	and 
	\[
	\text{hyponormal $\subseteq$ $\ast$-paranormal $\subseteq$ $n$-$\ast$-paranormal.}
	\]
	A closed subspace $M$ of $\mathcal{H}$ is said to be invariant under $T \in \mathcal{B(H)}$ if $T(M) \subseteq M$ and is said to be a reducing subspace if both $M$ and $M^{\perp}$ are invariant under $T$.
	
	Jab\l o\'nski, Jung and Stochel \cite{graph} have introduced weighted shifts on directed trees and study hyponormal and subnormal operators on directed trees. Recently backward shifts, Toeplitz operators and the operator $\nabla$ have been studied on directed tree setting (c.f \cite{graph1, gr2,gr4,gr3}).
	
	Let $G=(V,E)$ be a directed graph with vertex set $V$ and $E \subset V \times V$ be such that $(u,v)\in E$ means that $u$ is the parent of $v$ (denoted by $par(v)$) and $v$ is the child of $u$ (denoted by $chi(u)$). Note that  $chi^{(n)}(v)=chi(chi^{(n-1)}(v))$ and $par^{(n)}(v)=par(par^{(n-1)}(v))$. It is clear that $chi^{(0)}(v)=v=par^{(0)}(v)$. A vertex $v\in V$ is said to be a root of $G$ if there is no $(u,v)\in E$ for any $u \in V$. The set $V^0$ denotes $V \setminus root(G)$.
	
	Let $\ell^2(V)$ denote the Hilbert space of all square summable complex functions on $V$ with inner product 
	\[\langle f, g\rangle = \sum\limits_{u \in V} f(u)\overline{g(u)}; f,g \in \ell^2(V).\] 
	The set $\{e_u\}_{u\in V}$ where $e_u(v)=\begin{cases}
		1, & u=v\\ 0, & \text{otherwise}
	\end{cases}$ forms an orthonormal basis for $\ell^2(V)$. Let $\|\cdot \|_T$ and $\langle \cdot,\cdot \rangle_T$ denote the graph norm and graph inner product of $T$ defined by $\|f\|_T= \|f\|^2 +\|Tf\|^2$ and $\langle f,g \rangle_T = \langle f,g \rangle + \langle Tf,Tg \rangle$ respectively for every $f,g$ in domain of $T$. For $(\lambda_v)_{v\in V^0}$, the weighted shift $S_{\lambda}$ on $\ell^2(V)$ is given by,
	\[S_{\lambda} f(v)= \begin{cases}
		\lambda_v f(par(v)),& \text{if $v$ is not a root}\\ 0,& \text{if $v$ is a root}
	\end{cases}\] and the corresponding adjoint operator is defined as,
	\[S_{\lambda}^* f(v)= \sum\limits_{u\in chi(v)}\overline{\lambda_u}f(u), \ \ f\in \ell^2(V).\] For more details, refer \cite{graph}.
	
	Let $(V, \mathscr{V},\mu)$ be the measure space associated with the directed graph $G$ and positive measure $\mu$. Let $\psi: V \to V$ be a non-singular measurable function such that $\psi^{-n}(\mathscr{V}) \subseteq \mathscr{V}$ for $n \in \mathbb{N}$ and the measure $\mu \circ \psi^{-n}$ is absolutely continuous with respect to $\mu$. Let $h_n= \dfrac{d\mu \circ \psi^{-n}}{d\mu}$ denote the corresponding Radon-Nikodym derivative. Consider $L^2(V, \mathscr{V},\mu):=L^2(\mu)$ as the Hilbert space of all square integrable measurable functions and let $E_n(f|\psi^{-n}\mathscr{V}):=E_n(f)$ be the orthogonal projection onto the closed subspace $L^2(V, \psi^{-n}\mathscr{V},\mu)$ satisfying $\int_{S} f d\mu = \int_{S} E_n(f)d \mu$ for every $S\in \psi^{-n}\mathscr{V}$ called the conditional expectation operator. Since $\psi^{-n}\mathscr{V}$ is purely atomic,
	\[E_n(f)= \sum_{S_k} \dfrac{1}{\mu(S_k)}\left(\int_{S_k}f d\mu\right)\chi_{S_k}\]
	where $(S_k)_{k\in \mathbb{N}}$ be the disjoint collection of atoms that spans $\psi^{-n}\mathscr{V}$ and $\chi_{S_k}$ be the indicating function.
	
	Let $u: V\to \mathbb{C}$ be a measurable weight function. The weighted composition operator, $W:= W_{u,\psi^n}$ is defined as
	$(Wf)(x)= u(x)\cdot f(\psi^n(x))$ and the corresponding adjoint operator by $(W^*f)(x)= h_n(x)\cdot F(x)$ where $F \circ \psi^n = E_n(\bar{u}f)$, for every $f \in L^2(\mu)$. It is well known that $W^{*^n}W^n f= h_nE_n(|u|_n^{2})\circ \psi^{-n}f$ where $|u|_n^2= |u|^2\cdot |u\circ \psi|^2 \cdots |u \circ \psi^{n-1}|^2$ and $W^nW^{*^n}f= u_n\cdot h_n\circ \psi^n \cdot E_n(\bar{u_n}f)$ where $u_n= u\cdot u\circ \psi \cdots u \circ \psi^{n-1}$. Let $\psi(v)= par(v)$. 
	Let   $J_m=\{1,2,\cdots,m\}$,  $m\in \mathbb{N}$ and let $\eta_r\in\mathbb{N}\cup\{0\}$, $r\in J_m$.  Suppose that at least one of $\eta_r$ is non-zero for $r\in J_m$ and 
	
	\begin{equation} \label{eqn1}
		\nonumber V=\{v_1,v_2,\cdots,v_m\}\cup\bigcup_{r=1}^{m}\bigcup_{i=1}^{\eta_r} \{v^r_{i,j}: j\in\mathbb{N}\},
	\end{equation}
	where 
	$V_m=\{v_1,v_2,\cdots,v_m\}$ and $V_{\eta_r}=\bigcup_{i=1}^{\eta_r} \{v^r_{i,j}: j\in\mathbb{N}\}$ ($r\in J_\kappa$)  are  disjoint sets of distinct points of $V$.  For weighted composition operators, we consider $V$ as a  directed graph with one circuit $\{v_1,v_2,\cdots,v_k\}$, the set of  branching vertices in the one-circuit   
	and  $V_{\eta_r}$,  the set of branching elements for $r\in J_m$ where $\{v^r_{i,j}: j\in\mathbb{N}\}$ is the set of all vertices in the $i^{th} $ branch of  $v_r$ for $i\in J_{\eta_r}$ and  $\eta_r$ is the number of branches originating from the vertex  $v_r$. A more generalized version of this graph has been considered in \cite{graph b}. For $p\in \mathbb{N}$ and let $\Psi: \mathbb{Z}\rightarrow J_m$ be the unique function such that $p = k.m + \Psi(p)$. In \cite{m-isometry graph}, authors proved the following:
	
	\[\psi^p(v)= \begin{cases}
		v^r_{i,j} & \text{ if } v=v^r_{i,j+p} \text{ for }r\in J_\kappa i\in J_{\eta_r},\text{ and } j\in \mathbb{N} ,\\
		v_r& \text{ if } v=v^s_{i,j} \text{ for } s\in J_\kappa,  \text{ and } \Psi(p+r)=\Psi(s+j), j\in J_p, i\in J_{\eta_s},\\
		&\text{ or } v=v_{\Psi(p+r)}.
		
	\end{cases}
	\]
	
	\[\psi^{-p}(\{v\})= 
	\begin{cases}
		\{v^r_{i,j+p}\}& \text{ if } v=v^r_{i,j}, r\in J_\kappa , i\in J_{\eta_r},j\in \mathbb{N},\\
		\{v_{{\Psi}_(p+r)}\}\cup\bigcup\limits_{j=1}^p\bigcup\limits_{\substack{s=1\\   \Psi(p+r)=\Psi(s+j)}}^\kappa\bigcup\limits_{i=1}^{\eta_s}
		\{v^s_{i, j}\} &\text{ if } v=v_r,r\in J_\kappa 
	\end{cases}\]
	\[h_p(x)= \begin{cases}
		\frac{\mu(v^r_{i,j+p})}{\mu(v^r_{i, j})} & \text{ if } v=v^r_{i,j}, r\in J_m, i\in J_{\eta_r}, j\in \mathbb{N},\\
		\dfrac{\mu(v_{{\Psi}_(p+r)})+\sum\limits_{j=1}^p\sum\limits_{\substack{s=1\\ \Psi(p+r)=\Psi(s+j)} }^{m}
			\sum\limits_{i=1}^{\eta_{s}} \mu(v^s_{i,j})}{\mu(\{v_r\})}	 &\text{ if } v=v_r, r\in J_m
	\end{cases}.
	\]See \cite{m-isometry graph} for more details.
	
	Recently Pandey and Paulsen \cite{spectralcharacterisationofAN} proved that the spectrum of a positive $\mathcal{AN}$ operators contains atmost one eigenvalue with infinite multiplicity.  Later Venku Naidu and Ramesh characterised positive, self adjoint and normal absolutely norm attaining operators \cite{onabsnorm}. Bala and Ramesh characterised absolutely norm attaining hyponormal operators \cite{hyponormalabs}. Then absolutely norm attaining paranormal and $\ast$-paranormal operators were characterised by Ramesh and Bala respectively \cite{hyponormalabs, ANparanormal, Representationandnormalityof*-paranormalAN}. In this paper, we give examples of $\mathcal{AN}$ positive operators
	defined on $\ell^2(V)$ having more than one eigenvalue with infinite multiplicity in the spectrum. Also we study $\ast$-paranormal and quasi-$\ast$-paranormal operators on $\ell^2(V)$ setting and give some examples. We study some properties of absolutely norm attaining quasi-$\ast$-paranormal operators. In the last section, we study quasi-$\ast$-paranormal weighted composition operators and Lambert operators on graph settings and illustrate some examples. We also comment on norm attaining weighted composition operators.
	\section{$\mathcal{AN}$ $\ast$-paranormal weighted shift operators}
	In this section, we address the norm attaining  weighted shift operators defined on a directed tree. We further study norm attaining $\ast$-paranormal operators. 
	
	Consider the $\ast$-paranormal weighted shifts on directed graph settings.	
	\begin{theorem}
		Let $S_{\lambda}$ be the $\ast$-paranormal weighted shift on $\ell^2(V)$ with weights $(\lambda_v)$, then  	\[\left(\sum\limits_{u\in V} \left|\sum\limits_{v\in chi(u)}\overline{\lambda_v}f(v)\right|^2\right)^2 \leq \left(\sum\limits_{u\in V}\|S_{\lambda}^2e_u\|^2 |f(u)|^2\right)\left(\sum\limits_{u\in V}|f(u)|^2\right)\]
		for every $f \in \ell^2(V)$.
	\end{theorem}
	\begin{proof}
		Since $S_{\lambda}$ be $\ast$-paranormal,
		\[\|S_{\lambda}^*f\|^2 \leq \|S_{\lambda}^2f\|\|f\| \mbox{ for all }f \in \ell^2(V).\]
		We have,
		\[\|S_{\lambda}^*f\|^2 = \sum\limits_{u\in V} \left|\sum\limits_{v\in chi(u)}\overline{\lambda_v}f(v)\right|^2\]
		and 
		\[\|f\|^2 = \sum\limits_{u\in V}|f(u)|^2.\]
		Now,
		\begin{align*}
			S_{\lambda}^2f(u)&= S_{\lambda}\left( \lambda_u S_{\lambda}f(par(u))\right)\\
			& = \lambda_u \lambda_{par(u)} f(par^2(u))\\
			\sum\limits_{u\in V}	S_{\lambda}^2f(u)&= \sum\limits_{v\in chi(u)} \lambda_v \sum\limits_{w\in chi(v)} \lambda_w   f(u).
		\end{align*}
		Hence, 
		\[\|S_{\lambda}^2f\|^2 = \sum\limits_{u\in V}\|S_{\lambda}^2e_u\|^2 |f(u)|^2.\]
		Thus, 
		\[\left(\sum\limits_{u\in V} \left|\sum\limits_{v\in chi(u)}\overline{\lambda_v}f(v)\right|^2\right)^2 \leq \left(\sum\limits_{u\in V}\|S_{\lambda}^2e_u\|^2 |f(u)|^2\right)\left(\sum\limits_{u\in V}|f(u)|^2\right).\]	
	\end{proof}

	\begin{theorem}
		If $S_{\lambda}$ is a densely defined $\ast$-paranormal operator, then there exists $c>0$ such that
		$\sum\limits_{u\in chi(v)} \dfrac{|\lambda_u|^2}{1+\|S_{\lambda}^2e_u\|^2} \leq c$, \ for every $v\in V$.
	\end{theorem}
	\begin{proof}
		Since $S_{\lambda}$ is a densely defined $\ast$-paranormal operator, we have $D(S_{\lambda}^2)\subseteq D(S_{\lambda}^*)$. Using closed graph theorem to the identity embedding mapping $(D(S_{\lambda}^2),\|\cdot\|_{S_{\lambda}^2}) \to (D(S_{\lambda}^*),\|\cdot\|_{S_{\lambda}^*})$, we get a constant $c>0$ such that $\|f\|_{S_{\lambda}^*}^2 \leq c \|f\|_{S_{\lambda}^2}^2$. 
		That is,
		\[\sum\limits_{u\in V} \left|\sum\limits_{v\in chi(u)}\overline{\lambda_v}f(v)\right|^2 \leq c\sum\limits_{u\in V}(1+\|S_{\lambda}^2e_u\|^2)|f(u)|^2, \ \ f\in D(S_{\lambda}^2).\]
		Hence,
		\[\sum\limits_{u\in V} \left|\sum\limits_{v\in chi(u)}\overline{\lambda_v}f(v)\right|^2 \leq c\sum\limits_{u\in V} \sum\limits_{v\in chi(u)}(1+\|S_{\lambda}^2e_v\|^2)|f(v)|^2, \ \ f\in D(S_{\lambda}^2).\] 
		From the above inequality, we get
		\[\left|\sum\limits_{v\in chi(u)}\overline{\lambda_v}f(v)\right|^2 \leq c\sum\limits_{v\in chi(u)}(1+\|S_{\lambda}^2e_v\|^2)|f(v)|^2, \ \ f\in D(S_{\lambda}^2), \ \ u \in V.\] 
		Let $u \in V$ and $W \subset \{chi(u)\}$ . 
		
		Define $$f(v)=\begin{cases}
			\dfrac{\lambda_v}{1+\|S_{\lambda}^2e_v\|^2} & \text{if } v \in W, \\
			0 & \text{otherwise}
		\end{cases}$$.
		
		Then,
		\[\left(\sum\limits_{v\in W}\dfrac{|\lambda_v|^2}{1+\|S_{\lambda}^2e_v\|^2}\right)^2 \leq c\sum\limits_{v\in W}\dfrac{|\lambda_v|^2}{1+\|S_{\lambda}^2e_v\|^2}.\] Hence,
		\[\sum\limits_{v\in W}\dfrac{|\lambda_v|^2}{1+\|S_{\lambda}^2e_v\|^2} \leq c.\]
		Since $u$ is arbitrary and letting $W$ to $\{chi(u)\}$, we have 
		\[\sum\limits_{v\in chi(u)}\dfrac{|\lambda_v|^2}{1+\|S_{\lambda}^2e_v\|^2} \leq c\ \ u \in V.\]
	\end{proof}
	Now we comment on norm attaining weighted shifts on directed trees.
	\begin{theorem}\label{grapg char}
		Let $S_{\lambda}$ be a bounded weighted shift on $\ell^2(V)$ with weights $(\lambda_v)_{v\in V^0}$ where $V^0=V\setminus root(G)$. Then $S_{\lambda}$ is norm attaining if and only if there exists $v \in V$ such that \newline $\sum\limits_{u\in chi(v)}|\lambda_u|^2=\|S_{\lambda}\|^2$.
	\end{theorem}
	\begin{proof}
		Since $S_{\lambda} \in \mathcal{B}(\ell^2(V))$, we have $\|S_{\lambda}\|^2=\sup\limits_{u\in V}\|S_{\lambda}e_u\|^2$ \cite{graph}. This completes the proof.
	\end{proof}
	\begin{corollary}
		Let $S_{\lambda} \in \mathcal{B}(\ell^2(V))$ be norm attaining operator such that the sequence of weights $(\lambda_v)_{v\in V^0}$ has a limit $\lambda$. Then either $\lambda$ is a limit of a decreasing sequence or there exists a $\mu\in (\lambda_v)$ such that $\mu > \lambda$.
	\end{corollary}
	
	Consider the directed graph $G=(V,E)$ with weights $(\lambda_{ij})_{j=1}^\infty$, $i=1,2, \cdots,n$.
	
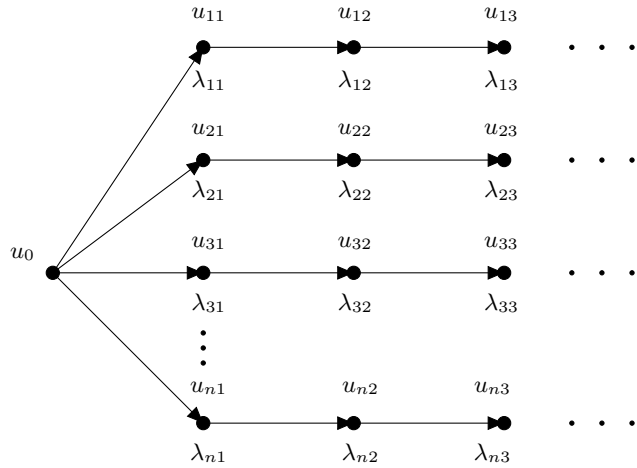
	\begin{figure}[h!]
		\begin{tikzpicture}[line cap=round,line join=round,>=triangle 45,x=.9cm,y=.9cm]
			\clip(-9.223030663277001,-2.540433489146846) rectangle (06.28333942857835,3.58275472164352);
			\draw [->,line width=0.4pt] (-5,0)-- (-3,3);
			\draw [->,line width=0.4pt] (-3,3)-- (-1,3);
			\draw [->,line width=0.4pt] (-1,3)-- (1,3);
			
			\draw [->,line width=0.4pt] (-5,0)-- (-3,1.5);
			\draw [->,line width=0.4pt] (-3,1.5)-- (-1,1.5);
			\draw [->,line width=0.4pt] (-1,1.5)-- (1,1.5);
			
			\draw [->,line width=0.4pt] (-5,0)-- (-3,0);
			\draw [->,line width=0.4pt] (-3,0)-- (-1,0);
			\draw [->,line width=0.4pt] (-1,0)-- (1,0);
			
			\draw [->,line width=0.4pt] (-5,0)-- (-3,-2);
			\draw [->,line width=0.4pt] (-3,-2)-- (-1,-2);
			\draw [->,line width=0.4pt] (-1,-2)-- (1,-2);
			\begin{scriptsize}
				\draw [fill=black] (-5,0) circle (2.5pt);
				\draw[color=black] (-5.401442887607267,0.2530088013793716) node {$u_0$};
				\draw [fill=black] (-3,3) circle (2.5pt);
				\draw[color=black] (-2.925823133557061,3.4363474094839224) node {$u_{11}$};
				\draw[color=black] (-2.925823133557061,2.563474094839224) node {$\lambda_{11}$};
				\draw [fill=black] (-1,3) circle (2.5pt);
				\draw[color=black] (-0.9789765308573843,3.4363474094839224) node {$u_{12}$};
				\draw[color=black] (-0.9789765308573843,2.563474094839224) node {$\lambda_{12}$};
				\draw [fill=black] (1,3) circle (2.5pt);
				\draw[color=black] (0.967870071842292,3.4363474094839224) node {$u_{13}$};
				\draw[color=black] (0.967870071842292,2.563474094839224) node {$\lambda_{13}$};
				\draw [fill=black] (1.9,3) circle (.8pt);
				\draw [fill=black] (2.3,3) circle (.8pt);
				\draw [fill=black] (2.7,3) circle (.8pt);
				
				\draw [fill=black] (1.9,1.5) circle (.8pt);
				\draw [fill=black] (2.3,1.5) circle (.8pt);
				\draw [fill=black] (2.7,1.5) circle (.8pt);
				
				\draw [fill=black] (-3,-2) circle (2.5pt);
				\draw[color=black] (-2.925823133557061,-1.529451225645034) node {$u_{n1}$};
				\draw[color=black] (-2.925823133557061,-2.39451225645034) node {$\lambda_{n1}$};
				\draw [fill=black] (-1,-2) circle (2.5pt);
				\draw[color=black] (-0.9068711011277666,-1.529451225645034) node {$u_{n2}$};
				\draw[color=black] (-0.9068711011277666,-2.39451225645034) node {$\lambda_{n2}$};
				\draw [fill=black] (1,-2) circle (2.5pt);
				\draw[color=black] (0.8476943556262626,-1.529451225645034) node {$u_{n3}$};
				\draw[color=black] (0.8476943556262626,-2.39451225645034) node {$\lambda_{n3}$}; 	
				\draw [fill=black] (1.9,-2) circle (.8pt);
				\draw [fill=black] (2.3,-2) circle (.8pt);
				\draw [fill=black] (2.7,-2) circle (.8pt);
				
				\draw [fill=black] (-3,0) circle (2.5pt);
				\draw[color=black] (-2.925823133557061,.4) node {$u_{31}$};
				\draw[color=black] (-2.925823133557061,-.4) node {$\lambda_{31}$};
				\draw [fill=black] (-1,0) circle (2.5pt);
				\draw[color=black] (-0.9789765308573843,.4) node {$u_{32}$};
				\draw[color=black] (-0.9789765308573843,-.4) node {$\lambda_{32}$};
				\draw [fill=black] (1,0) circle (2.5pt);
				\draw[color=black] (0.967870071842292,.4) node {$u_{33}$};
				\draw[color=black] (0.967870071842292,-.4) node {$\lambda_{33}$};
				\draw [fill=black] (1.9,0) circle (.8pt);
				\draw [fill=black] (2.3,0) circle (.8pt);
				\draw [fill=black] (2.7,0) circle (.8pt);
				
				\draw [fill=black] (-3,1.5) circle (2.5pt);
				\draw[color=black] (-2.925823133557061,1.9) node {$u_{21}$};
				\draw[color=black] (-2.925823133557061,1.1) node {$\lambda_{21}$};
				\draw [fill=black] (-1,1.5) circle (2.5pt);
				\draw[color=black] (-0.9789765308573843,1.9) node {$u_{22}$};
				\draw[color=black] (-0.9789765308573843,1.1) node {$\lambda_{22}$};
				\draw [fill=black] (1,1.5) circle (2.5pt);
				\draw[color=black] (0.967870071842292,1.9) node {$u_{23}$};
				\draw[color=black] (0.967870071842292,1.1) node {$\lambda_{23}$};
				
				\draw [fill=black] (-3,-1) circle (.8pt);
				\draw [fill=black] (-3,-1.2) circle (.8pt);
				\draw [fill=black] (-3,-.8) circle (.8pt);
			\end{scriptsize}
		\end{tikzpicture}
		\caption{Directed graph with n branches}\label{branch}
	\end{figure}
	Let $f \in \ell^2(V)$, then $f=  \alpha_0 e_0 + \sum\limits_{\substack{j=1 \\ 1\leq i \leq n}}^\infty \alpha_{ij}e_{ij}$ where $\alpha_0, \alpha_{ij}\in \mathbb{C}$. 
	
	Now,
	\begin{align*}
		S_{\lambda}f &= \alpha_0 \left(\sum\limits_{i=1}^n \lambda_{i1} e_{i1}\right) + \sum\limits_{\substack{j=1 \\ 1\leq i \leq n}}^\infty \alpha_{ij}\lambda_{i,j+1} e_{i,j+1}.\\
		\|S_{\lambda}f\|^2 &=|\alpha_0|^2 \left(\sum\limits_{i=1}^n |\lambda_{i1}|^2 \right) + \sum\limits_{\substack{j=1 \\ 1\leq i \leq n}}^\infty |\alpha_{ij}|^2|\lambda_{i,j+1}|^2. 
	\end{align*}
	
	Using Theorem~\ref{grapg char}, there exists some $e_v$ where $v \in V$ such that $\|S_{\lambda}\|^2 = \|S_{\lambda}e_v\|^2$. Hence, 
	\[\|S_{\lambda}\|^2= \max\left\{ \sum\limits_{i=1}^n |\lambda_{i1}|^2, |\lambda_{ij}|^2;1 \leq i \leq n,j>1\right\}.\]
	In particular, consider the directed graph $G=(V,E)$ with weights $(\lambda_{ij})_{j=1}^\infty$ and $n=3$.
	
	Let $f \in \ell^2(V)$. Then $f= \alpha_0 e_0 + \sum\limits_{j=1}^\infty \alpha_{1j}e_{1j} + \sum\limits_{j=1}^\infty \alpha_{2j}e_{2j} +\sum\limits_{j=1}^\infty \alpha_{3j}e_{3j}$ where $\alpha_0, \alpha_{ij}$'s $\in \mathbb{C}, e_0= e_{u_0}$ and $e_{ij}=e_{u_{ij}}$. Hence $\|f\|^2= |\alpha_0|^2 +\sum\limits_{j=1}^\infty |\alpha_{1j}|^2+\sum\limits_{j=1}^\infty |\alpha_{2j}|^2+\sum\limits_{j=1}^\infty |\alpha_{3j}|^2$.
	
	Now,
	\begin{align*}
		S_{\lambda}f
		&=\alpha_0 S_{\lambda}(e_0) + \sum\limits_{j=1}^\infty \alpha_{1j}S_{\lambda}(e_{1j}) + \sum\limits_{j=1}^\infty \alpha_{2j}S_{\lambda}(e_{2j}) + \sum\limits_{j=1}^\infty \alpha_{3j}S_{\lambda}(e_{3j}) \\
		& = \alpha_0(\lambda_{11}e_{11}+\lambda_{21}e_{21}+\lambda_{31}e_{31})+ \sum\limits_{j=1}^\infty \alpha_{1j}\lambda_{1,j+1}e_{1,j+1} + \sum\limits_{j=1}^\infty \alpha_{2j}\lambda_{2,j+1}e_{2,j+1}  +\\& \ \ \  \sum\limits_{j=1}^\infty \alpha_{3j}\lambda_{3,j+1}e_{3,j+1}. 
	\end{align*}
	Therefore, 
	\[\|S_{\lambda}f\|^2=|\alpha_0|^2(|\lambda_{11}|^2+|\lambda_{21}|^2+|\lambda_{31}|^2) +\sum\limits_{j=1}^\infty |\alpha_{1j}|^2|\lambda_{1,j+1}|^2+\sum\limits_{j=1}^\infty |\alpha_{2j}|^2|\lambda_{2,j+1}|^2 +\sum\limits_{j=1}^\infty |\alpha_{3j}|^2|\lambda_{3,j+1}|^2.\]
	Now, $S_{\lambda}$ is norm attaining if and only if there exists an $f\in \ell^2(V)$ with $\|f\|^2= |\alpha_0|^2 +\sum\limits_{j=1}^\infty |\alpha_{1j}|^2+\sum\limits_{j=1}^\infty |\alpha_{2j}|^2 +\sum\limits_{j=1}^\infty |\alpha_{3j}|^2 =1$ such that \[\|S_{\lambda}f\|^2= \|S_{\lambda}\|^2= \sup\limits_{g\in \ell^2(V)}\dfrac{\|S_{\lambda}g\|^2}{\|g\|^2}.\]
	By Theorem~\ref{grapg char}, there exists some $e_v$ where $v \in V$ such that $\|S_{\lambda}\|^2 = \|S_{\lambda}e_v\|^2$. Thus, 
	\[\|S_{\lambda}\|^2= \max\{ |\lambda_{11}|^2+|\lambda_{21}|^2  +|\lambda_{31}|^2, |\lambda_{ij}|,i=1,2,3;j>1\}\]
	Now consider the root free directed graph with a cycle with $m$ vertices, say $\{v_1.v_2, \cdots, v_m\}$ such that $par(v_1)=v_{m}$ and $par(v_{i+1})=v_i$ for $2 \leq i \leq m$. Also fix $n_i$ branches at each vertex $v_i$ as in Figure~\ref{branch}.
	 Now we have $V=\{v_1, v_2, \cdots, v_m\}\cup \cup_{s=1}^m \cup_{i=1}^{n_s} \cup_{j=1}^\infty \{v_{i,j}^s\}$ and $\{e_n\}_{n=1}^m\cup \{e_{i,j}^s\}$ be the basis for $\ell^2(V)$. Then for $f \in \ell^2(V)$ be defined as,
	 $f= \sum\limits_{s=1}^m \alpha_s e_s + \sum\limits_{i=1}^{n_s} \sum\limits_{j=1}^{\infty}\alpha_{i,j}^s e_{i,j}^s.$
	 Therefore, 
	 \[S_\lambda f = \sum\limits_{s=1}^{m-1} \alpha_s \left(\lambda_{s+1}e_{s+1}+\sum\limits_{i=1}^{n_s}\lambda_{i,1}^se_{i,1}^s\right)+\alpha_m\left(\lambda_{1}e_{1}+\sum\limits_{i=1}^{n_s}\lambda_{i,1}^me_{i,1}^m\right)+ \sum\limits_{s=1}^{m}\sum\limits_{\substack{i=1\\1 \leq j \leq \infty}}^{n_s}  \alpha_{i,j}^s\lambda_{i,j+1}^se_{i,j+1}^s\]
	 and 
	 \[\|S_\lambda f\|^2 =\sum\limits_{s=1}^{m-1} |\alpha_s|^2 \left(|\lambda_{s+1}|^2+\sum\limits_{i=1}^{n_s}|\lambda_{i,1}^s|^2\right)+|\alpha_m|^2\left(|\lambda_{1}|^2+\sum\limits_{i=1}^{n_s}|\lambda_{i,1}^m|^2\right)+ \sum\limits_{s=1}^{m}\sum\limits_{\substack{i=1\\1 \leq j \leq \infty}}^{n_s}  |\alpha_{i,j}^s|^2|\lambda_{i,j+1}^s|^2.\]
	 Using Theorem~\ref{grapg char}, we have
	 \[\|S_\lambda\|^2= \max \left\{|\lambda_{s+1}|^2+\sum\limits_{i=1}^{n_s}|\lambda_{i,1}^s|^2,|\lambda_{1}|^2+\sum\limits_{i=1}^{n_s}|\lambda_{i,1}^m|^2,|\lambda_{i,j+1}^{s+1}|^2: 1 \leq s<m, 1\leq i \leq n_s, 1\leq j \leq \infty\right\}.\]
	 \begin{example}
	 	Consider a directed graph with $m=2, n_1=3$ and $n_2=0$.\\ Let $\lambda=\begin{cases}
	 		2 & \text{ if } \lambda=\lambda_{1}, \lambda_{2}\\
	 		1 & \text{ otherwise} 
	 	\end{cases}.$ Then $\|S_\lambda\|=\sqrt{7}=\|S_\lambda(e_1)\|$. Hence
	 	$S_\lambda$ is norm attaining.
	 \end{example}
	\begin{theorem}
		Let $S_{\lambda}$ be a weighted shift on $\ell^2(V)$ with weights $(\lambda_v)_{v\in V^0}$ where $V^0=V\setminus root(G)$. Then $S_{\lambda}$ is $\ast$-paranormal if and only if 
		\[\| S_{\lambda}^2e_v\|^2-2k \|S_{\lambda}^*e_v\|^2+k^2 \geq 0 \]
		for all $k>0, v\in V$.
	\end{theorem}
	\begin{proof}
		Since $S_{\lambda}$ is $\ast$-paranormal,
		\[\|S_{\lambda}^*f\|^2 \leq \|S_{\lambda}^2f\|\|f\| \mbox{ for all }f \in \ell^2(V).\] Let $f=e_v$, then
		$$\|S_{\lambda}^*e_v\|^2 \leq \|S_{\lambda}^2e_v\|\|e_v\| $$
		That is, $$(2\|S_{\lambda}^*e_v\|)^2 \leq 4\|S_{\lambda}^2e_v\|\|e_v\|.$$ Hence,
		\[\| S_{\lambda}^2e_v\|^2-2k \|S_{\lambda}^*e_v\|^2+k^2 \geq 0 \]
		for all $k>0, v\in V$.
	\end{proof}
	\begin{corollary}\label{gr*para}
		Let $S_{\lambda}$ be a $\ast$-paranormal weighted shift on $\ell^2(V)$ with weights $(\lambda_v)_{v\in V^0}$ such that $S_{\lambda}^*$ is norm attaining. Then there exists a $v \in V$ such that $\|S_{\lambda}^2e_v\| = \|S_{\lambda}\|^2$.
	\end{corollary}
	\begin{proof}
		Since $S_{\lambda}$ is a $\ast$-paranormal weighted shift on $\ell^2(V)$ we have,
		\[\| S_{\lambda}^2e_v\|^2-2k \|S_{\lambda}^*e_v\|^2+k^2 \geq 0 \]
		for all $k>0, v\in V$. Since $S_{\lambda}^*$ is norm attaining, there exists $e_v \in \ell^2(V)$ such that \mbox{$\|S_{\lambda}^*e_v\|=\|S_{\lambda}\|$}. Let $k=\|S_{\lambda}\|^2$. Then we have,
		\[\|S_{\lambda}^2e_v\|^2 \geq \|S_{\lambda}\|^4.\]
		That is, $\|S_{\lambda}^2e_v\| \geq \|S_{\lambda}\|^2.$ Hence $\|S_{\lambda}^2e_v\| = \|S_{\lambda}\|^2$.
	\end{proof}
	\begin{corollary}
		Let $S_{\lambda}$ be a $\ast$-paranormal weighted shift on $\ell^2(V)$  such that $S_{\lambda}^*$ is norm attaining, then $S_{\lambda}^2$ is norm attaining.
	\end{corollary}
	\begin{proof}
		From Corollary~\ref{gr*para}, we have $\|S_{\lambda}\|^2=\|S_{\lambda}^2e_v\| \leq \|S_{\lambda}^2\| \leq \|S_{\lambda}\|^2.$ 
	\end{proof}
	\begin{remark}
		Let $S_{\lambda}$ be absolutely norm attaining $\ast$-paranormal weighted shift on $\ell^2(V)$ with weights $(\lambda_v)_{v\in V^0}$, then $S_{\lambda}^2$ is also absolutely norm attaining.
	\end{remark}
	\begin{example}\label{ex}
		Consider the directed graph $G=(V,E)$ with weights $\lambda_{ij}=\begin{cases}
			\frac{1}{9}, & i=j=1\\
			\frac{1}{3},& i=2,j=1\\
			\frac{1}{4}, & i=j=2\\
			2, & \text{otherwise}
		\end{cases}$.
		
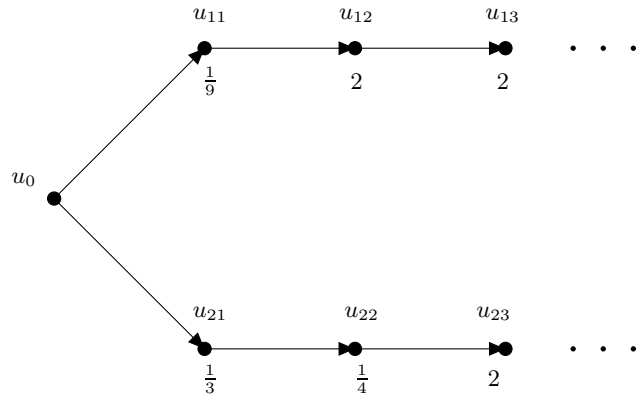
\begin{figure}[h!]
			\begin{tikzpicture}[line cap=round,line join=round,>=triangle 45,x=.7cm,y=.7cm]
				\clip(-9.323030663277001,-2.640433489146846) rectangle (08.28333942857835,2.58275472164352);
				\draw [->,line width=0.4pt] (-5,0)-- (-3,2);
				\draw [->,line width=0.4pt] (-3,2)-- (-1,2);
				\draw [->,line width=0.4pt] (-1,2)-- (1,2);
				\draw [->,line width=0.4pt] (-5,0)-- (-3,-2);
				\draw [->,line width=0.4pt] (-3,-2)-- (-1,-2);
				\draw [->,line width=0.4pt] (-1,-2)-- (1,-2);
				\begin{scriptsize}
					\draw [fill=black] (-5,0) circle (2.5pt);
					\draw[color=black] (-5.401442887607267,0.2530088013793716) node {$u_0$};
					\draw [fill=black] (-3,2) circle (2.5pt);
					\draw[color=black] (-2.925823133557061,2.4363474094839224) node {$u_{11}$};
					\draw[color=black] (-2.925823133557061,1.563474094839224) node {$\frac{1}{9}$};
					\draw [fill=black] (-1,2) circle (2.5pt);
					\draw[color=black] (-0.9789765308573843,2.4363474094839224) node {$u_{12}$};
					\draw[color=black] (-0.9789765308573843,1.563474094839224) node {$2$};
					\draw [fill=black] (1,2) circle (2.5pt);
					\draw[color=black] (0.967870071842292,2.4363474094839224) node {$u_{13}$};
					\draw[color=black] (0.967870071842292,1.563474094839224) node {$2$};
					\draw [fill=black] (1.9,2) circle (.8pt);
					\draw [fill=black] (2.3,2) circle (.8pt);
					\draw [fill=black] (2.7,2) circle (.8pt);
					\draw [fill=black] (-3,-2) circle (2.5pt);
					\draw[color=black] (-2.925823133557061,-1.529451225645034) node {$u_{21}$};
					\draw[color=black] (-2.925823133557061,-2.39551225645034) node {$\frac{1}{3}$};
					\draw [fill=black] (-1,-2) circle (2.5pt);
					\draw[color=black] (-0.9068711011277666,-1.529451225645034) node {$u_{22}$};
					\draw[color=black] (-0.9068711011277666,-2.39451225645034) node {$\frac{1}{4}$};
					\draw [fill=black] (1,-2) circle (2.5pt);
					\draw[color=black] (0.8476943556262626,-1.529451225645034) node {$u_{23}$};
					\draw[color=black] (0.8476943556262626,-2.39451225645034) node {$2$}; 	
					\draw [fill=black] (1.9,-2) circle (.8pt);
					\draw [fill=black] (2.3,-2) circle (.8pt);
					\draw [fill=black] (2.7,-2) circle (.8pt);
				\end{scriptsize}
			\end{tikzpicture}
			\caption{Directed tree with two branches on the root}	
		\end{figure}
		\begin{align*}
			S_{\lambda}^*f &=\left(\dfrac{1}{9}\alpha_{11}+\dfrac{1}{3}\alpha_{21}\right)e_0+2\sum\limits_{j=2}^\infty \alpha_{1j}e_{1,j-1} + \dfrac{1}{4}\alpha_{22}e_{21}+2\sum\limits_{j=3}^\infty \alpha_{2j}e_{2,j-1}
		\end{align*}
		\begin{align*}
			S_{\lambda}f &= \alpha_0(\lambda_{11}e_{11}+\lambda_{21}e_{21})+ \sum\limits_{j=1}^\infty \alpha_{1j}\lambda_{1,j+1}e_{1,j+1} + \sum\limits_{j=1}^\infty \alpha_{2j}\lambda_{2,j+1}e_{2,j+1}\\
			& = \alpha_0\left(\dfrac{1}{9}e_{11}+\dfrac{1}{3}e_{21}\right)+ 2\sum\limits_{j=1}^\infty \alpha_{1j}e_{1,j+1} + \dfrac{1}{4}\alpha_{21}e_{22}+2 \alpha_{22}e_{23}+2\sum\limits_{j=3}^\infty \alpha_{2j}e_{2,j+1}\\
			S_{\lambda}^2f &=\alpha_0 \left(\dfrac{2}{9}e_{12}+\dfrac{1}{12}e_{22}\right) + 4\sum\limits_{j=1}^\infty \alpha_{1j}e_{1,j+2}+\dfrac{1}{2}\alpha_{21}e_{23}+4 \alpha_{22}e_{24}+4\sum\limits_{j=3}^\infty \alpha_{2j}e_{2,j+2}.
		\end{align*}
		Therefore, 
		\begin{align*}
			\|S_{\lambda}^*f\|^2&= \left|\dfrac{1}{9}\alpha_{11}+\dfrac{1}{3}\alpha_{21}\right|^2+ 4\sum\limits_{j=2}^\infty |\alpha_{1j}|^2 + \dfrac{1}{16}|\alpha_{22}|^2+4\sum\limits_{j=3}^\infty |\alpha_{2j}|^2\\
			\|S_{\lambda}^2f \|^2 &=\dfrac{73}{1296}|\alpha_0|^2+16|\alpha_{11}|^2+16\sum\limits_{j=2}^\infty |\alpha_{1j}|^2+\dfrac{1}{4}|\alpha_{21}|^2+16 |\alpha_{22}|^2 +16\sum\limits_{j=3}^\infty |\alpha_{2j}|^2.
		\end{align*}
		Let $|\alpha_0|^2=x, |\alpha_{11}|^2=y, \sum\limits_{j=2}^\infty |\alpha_{1j}|^2 =z,|\alpha_{21}|^2=p,|\alpha_{22}|^2=q $ and $\sum\limits_{j=3}^\infty |\alpha_{2j}|^2=r$, we get $\|f\|^2=x+y+z+p+q+r$ and $\|S_{\lambda}^*f\|^4 \leq \|S_{\lambda}^2f \|^2\|f\|^2$ for every $f \in \ell^2(V)$ and $x,y,z,p,q,r\geq 0$. That is, $S_{\lambda}$ is $\ast$-paranormal. Also we observe that $\dfrac{1}{3}=\|S_{\lambda}^*e_{21}\|> \|S_{\lambda}e_{21}\|=\dfrac{1}{4}$, hence $S_{\lambda}$ is not hyponormal.
		
		Now we show that $S_{\lambda}$ is norm attaining on $\ell^2(V)$. We have 
		\begin{align*}
			\|S_{\lambda}f\|^2&=|\alpha_0|^2\left(\dfrac{1}{81}+\dfrac{1}{9}\right)+ 4\sum\limits_{j=1}^\infty |\alpha_{1j}|^2 + \dfrac{1}{16}|\alpha_{21}|^2+4\sum\limits_{j=2}^\infty |\alpha_{2j}|^2
		\end{align*}
			\begin{align*}
			\|S_{\lambda}\|^2&= \sup\limits_{\substack{f\in \ell^2(V)\\ \|f\|=1 }}\left( |\alpha_0|^2\left(\dfrac{1}{81}+\dfrac{1}{9}\right)+ 4\sum\limits_{j=1}^\infty |\alpha_{1j}|^2 + \dfrac{1}{16}|\alpha_{21}|^2+4\sum\limits_{j=2}^\infty |\alpha_{2j}|^2\right)\\
			&=\sup\limits_{\substack{f\in \ell^2(V)\\ \|f\|=1 }}\left(4\left( |\alpha_0|^2 +\sum\limits_{j=1}^\infty |\alpha_{1j}|^2+\sum\limits_{j=1}^\infty |\alpha_{2j}|^2 \right)- \dfrac{314}{81}|\alpha_0|^2 - \dfrac{63}{81}|\alpha_{21}|^2 \right)\\
			&=\sup\limits_{\substack{f\in \ell^2(V)\\ \|f\|=1 }} \left(4\|f\|^2- \dfrac{314}{81}|\alpha_0|^2 - \dfrac{63}{81}|\alpha_{21}|^2\right) \\
			&=4.
		\end{align*}
		That is, $\|S_{\lambda}\|=2$. Also we have $\|S_{\lambda}e_{11}\|= \|2e_{12}\|=2=\|S_{\lambda}\|$. Therefore $S_{\lambda}$ is a norm attaining $\ast$-paranormal operator on $\ell^2(V)$.
	\end{example}
	From the above example we have $\|S_{\lambda}^2e_{11}\| =\|2S_{\lambda}e_{11}\|=\|S_{\lambda}\|\|S_{\lambda}e_{12}\|=\|S_{\lambda}\|\|S_{\lambda}e_{11}\|$. Hence, $\overline{span\{Des(e_{11})\}}$ is an invariant subspace for $S_{\lambda}$ in $\ell^2(V)$.
	\section{Quasi-$\ast$-paranormal weighted shifts on directed trees}
	In this section we study quasi-$\ast$-paranormal weighted shifts on directed graphs.
	\begin{theorem}
		Let $S_{\lambda}$ be a quasi-$\ast$-paranormal weighted shift on $\ell^2(V)$ with weights $(\lambda_v)$, then 
		\[\left(\sum\limits_{u\in V} \|S_{\lambda}e_u\|^4|f(u)|^2\right)^2 \leq \left(\sum\limits_{u\in V}\|S_{\lambda}^3e_u\|^2 |f(u)|^2\right)\left(\sum\limits_{u\in V} \|S_{\lambda}e_u\|^2 |f(u)|^2\right).\]
	\end{theorem}
	\begin{proof}
		Since $S_{\lambda}$ be a quasi-$\ast$-paranormal,
		\[\|S_{\lambda}^*S_{\lambda}f\|^2 \leq \|S_{\lambda}^3f\|\|S_{\lambda}f\| \mbox{ for all }f \in \ell^2(V).\]
		We have,
		\[\|S_{\lambda}f\|^2 = \sum\limits_{u\in V} \|S_{\lambda}e_u\|^2 |f(u)|^2\]
		and
		\[\|S_{\lambda}^*S_{\lambda}f\|^2 = \sum\limits_{u\in V} \|S_{\lambda}e_u\|^4|f(u)|^2.\]
		Now,
		\begin{align*}
			S_{\lambda}^3f(u)&= S_{\lambda}\left( \lambda_u S_{\lambda}^2f(par(u))\right)\\
			& =S_{\lambda}\left( \lambda_u \lambda_{par(u)}S_{\lambda}f(par^2(u))\right)\\
			&= \lambda_u \lambda_{par(u)} \lambda_{par^2(u)}f(par^3(u))\\
			\sum\limits_{u\in V}	S_{\lambda}^3f(u)&= \sum\limits_{v\in chi(u)} \lambda_v \sum\limits_{w\in chi(v)} \lambda_w \sum\limits_{t\in chi(w)} \lambda_t f(u).
		\end{align*}
		Hence, 
		\[\|S_{\lambda}^3f\|^2 = \sum\limits_{u\in V}\|S_{\lambda}^3e_u\|^2 |f(u)|^2.\]
		Thus, 
		\[\left(\sum\limits_{u\in V} \|S_{\lambda}e_u\|^4|f(u)|^2\right)^2 \leq \left(\sum\limits_{u\in V}\|S_{\lambda}^3e_u\|^2 |f(u)|^2\right)\left(\sum\limits_{u\in V} \|S_{\lambda}e_u\|^2 |f(u)|^2\right).\]	
	\end{proof}
	\begin{theorem}
		Let $S_{\lambda}$ be a weighted shift on $\ell^2(V)$ with weights $(\lambda_v)_{v\in V^0}$ where \mbox{$V^0=V\setminus root(G)$}. Then $S_{\lambda}$ is quasi-$\ast$-paranormal if and only if 
		\[\|S_{\lambda}^3e_v\|^2 -2k \|S_{\lambda}e_v\|^4+k^2 \|S_{\lambda}e_v\|^2 \geq 0\]
		for all $k>0, v\in V.$
	\end{theorem}
	\begin{proof}
		Let Let $S_{\lambda}$ be quasi-$\ast$-paranormal, then
		\[S_{\lambda}^{*3}S_{\lambda}^3-2k (S_{\lambda}^*S_{\lambda})^2+k^2  S_{\lambda}^*S_{\lambda} \geq 0 \mbox{ for all } k>0.\]
		Now we have,
		\begin{align*}
			S_{\lambda}^{*3}S_{\lambda}^3e_v& = \|S_{\lambda}^3e_v\|^2e_v\\
			S_{\lambda}^*S_{\lambda}e_v& = \|S_{\lambda}e_v\|^2e_v\\
			(S_{\lambda}^*S_{\lambda})^2e_v& = \|S_{\lambda}e_v\|^4e_v.		
		\end{align*}
		So we get,
		\[\langle \|S_{\lambda}^3e_v\|^2e_v-2k \|S_{\lambda}e_v\|^4e_v +k^2 \|S_{\lambda}e_v\|^2e_v, e_v \rangle  \geq 0\]
		and hence \[\|S_{\lambda}^3e_v\|^2 -2k \|S_{\lambda}e_v\|^4+k^2 \|S_{\lambda}e_v\|^2 \geq 0 \mbox{ for all } k>0, v\in V.\]
		\noindent
		For  the converse part, let $\|S_{\lambda}^3e_v\|^2 -2k \|S_{\lambda}e_v\|^4+k^2 \|S_{\lambda}e_v\|^2 \geq 0$ for all $k>0, v\in V.$ We have if $f \in \ell^2(V)$, then $f= \sum\limits_{v\in V} \alpha_ve_v$.
		Now,
		\begin{align*}
			S_{\lambda}^{*3}S_{\lambda}^3f-2k (S_{\lambda}^*S_{\lambda})^2f+k^2  S_{\lambda}^*S_{\lambda} f &= 	S_{\lambda}^{*3}S_{\lambda}^3\sum\limits_{v\in V} \alpha_ve_v-2k (S_{\lambda}^*S_{\lambda})^2\sum\limits_{v\in V} \alpha_ve_v+k^2  S_{\lambda}^*S_{\lambda} \sum\limits_{v\in V} \alpha_ve_v \\
			&=\sum\limits_{v\in V} \alpha_v (S_{\lambda}^{*3}S_{\lambda}^3-2k (S_{\lambda}^*S_{\lambda})^2+k^2  S_{\lambda}^*S_{\lambda})e_v\\
			&= \sum\limits_{v\in V} \alpha_v\|S_{\lambda}^3e_v\|^2e_v-2k \|S_{\lambda}e_v\|^4e_v +k^2 \|S_{\lambda}e_v\|^2e_v\\
			\langle S_{\lambda}^{*3}S_{\lambda}^3f-2k (S_{\lambda}^*S_{\lambda})^2f+k^2  S_{\lambda}^*S_{\lambda} f,f \rangle &= \sum\limits_{v\in V} \alpha_v \langle \|S_{\lambda}^3e_v\|^2e_v-2k \|S_{\lambda}e_v\|^4e_v +k^2 \|S_{\lambda}e_v\|^2e_v, e_v \rangle \\
			& \geq 0.
		\end{align*}
		Hence $S_{\lambda}$ is quasi-$\ast$-paranormal. 
	\end{proof}
	\begin{corollary}
		Let $S_{\lambda}$ be a norm attaining quasi-$\ast$-paranormal weighted shift on $\ell^2(V)$ with weights $(\lambda_v)_{v\in V^0}$. Then there exists a $v \in V$ such that $\|S_{\lambda}^3e_v\| = \|S_{\lambda}\|^3.$
	\end{corollary}
	\begin{proof}
		Since $S_{\lambda}$ is a quasi-$\ast$-paranormal weighted shift on $\ell^2(V)$ we have,
		\[\|S_{\lambda}^3e_v\|^2 -2k \|S_{\lambda}e_v\|^4+k^2 \|S_{\lambda}e_v\|^2 \geq 0\]
		for all $k>0, v\in V$. Since $S_{\lambda}$ is norm attaining, there exists $e_v \in \ell^2(V)$ such that \mbox{$\|S_{\lambda}e_v\|=\|S_{\lambda}\|$}. Fix $k=\|S_{\lambda}\|^2$ and simplifying, we get 
		\[\|S_{\lambda}^3e_v\|^2 \geq \|S_{\lambda}\|^6.\]
		That is, $\|S_{\lambda}^3e_v\| =\|S_{\lambda}\|^3.$
	\end{proof}
	\begin{corollary}
		Let $S_{\lambda}$ be a norm attaining quasi-$\ast$-paranormal weighted shift on $\ell^2(V)$ with weights $(\lambda_v)_{v\in V^0}$. Then $S_{\lambda}^3$ is norm attaining.
	\end{corollary}
	\begin{proof}
		Since $S_{\lambda}$ be a norm attaining quasi-$\ast$-paranormal, there exists an $e_v \in \ell^2(V)$ such that $\|S_{\lambda}\|^3=\|S_{\lambda}^3e_v\| \leq \|S_{\lambda}^3\| \leq  \|S_{\lambda}\|^3$.
	\end{proof}
	\begin{remark}
		Let $S_{\lambda}$ be absolutely norm attaining quasi-$\ast$-paranormal weighted shift on $\ell^2(V)$ with weights $(\lambda_v)_{v\in V^0}$, then $S_{\lambda}^3$ is also absolutely norm attaining.
	\end{remark}
	
	\begin{theorem}
		Let $S_{\lambda}$ be a quasi-$\ast$-paranormal weighted shift on $\ell^2(V)$ with weights $(\lambda_v)_{v\in V^0}$ where $V^0=V\setminus root(G)$ then for $f\in \ell^2(V)$, there exists $c>0$ such that 
		\begin{enumerate}
			\item 
			$\sum\limits_{u \in V}\|S_{\lambda}e_u\|^4|f(u)|^2 \leq c\sum\limits_{u \in V}(1+\|S_{\lambda}e_u\|^2)|f(u)|^2 $
			\item 
			$\sum\limits_{u \in V}\|S_{\lambda}e_u\|^4|f(u)|^2 \leq c\sum\limits_{u \in V}(1+\|S_{\lambda}^3e_u\|^2)|f(u)|^2 $.
		\end{enumerate}
	\end{theorem}
	\begin{proof}
		Since $S_{\lambda}$ is quasi-$\ast$-paranormal, we have $D(S_{\lambda}^3)\subseteq D(S_{\lambda}^*S_{\lambda})$ and $D(S_{\lambda})\subseteq D(S_{\lambda}^*S_{\lambda})$.
		Hence from  closed graph theorem,  there exists a $c>0$ such that \[\|f\|_{S_{\lambda}^*S_{\lambda}}^2 \leq c\|f\|_{S_{\lambda}^3}^2. \]
		Since $\|f\|_T^2 = \|f\|^2+\|Tf\|^2$, we get
		\[\sum\limits_{u \in V}\|S_{\lambda}e_u\|^4|f(u)|^2 \leq c\sum\limits_{u \in V}(1+\|S_{\lambda}e_u\|^2)|f(u)|^2 .\]
		In a similar way we can prove (2).
	\end{proof}
	\begin{example}
		Consider the directed graph as in Example~\ref{ex} with weights,\\ $\lambda_{ij}=\begin{cases}
			1 & \mbox{ if $j=1$}\\
			2  & \mbox{ if $i=1$ and $j\geq2$}\\
			2 & \mbox{ if $i=j=2$}\\
			1  & \mbox{ if $i=2$ and $j=3$}\\
			4 & \mbox{ if $i=2$ and $j\geq4$}
		\end{cases}$. 
		
		Then,
		\begin{align*}
			&S_{\lambda}f&=&\alpha_0(e_{11}+e_{21})+2\sum\limits_{j=1}^{\infty}\alpha_{1j}e_{1,j+1} + 2\alpha_{21}e_{22}+\alpha_{22}e_{23}+4\sum\limits_{j=3}^{\infty}\alpha_{2j}e_{2,j+1} \\
			&S_{\lambda}^*S_{\lambda}f &=& 2\alpha_0e_0+4\sum\limits_{j=1}^{\infty}\alpha_{1j}e_{1,j} + 4\alpha_{21}e_{21}+\alpha_{22}e_{22}+16\sum\limits_{j=3}^{\infty}\alpha_{2j}e_{2,j} \\
			&S_{\lambda}^3f&=&\alpha_0(4e_{13}+2e_{23})+8\sum\limits_{j=1}^{\infty}\alpha_{1j}e_{1,j+3} + 8\alpha_{21}e_{24}+16\alpha_{22}e_{25}+64\sum\limits_{j=3}^{\infty}\alpha_{2j}e_{2,j+3} \\
			&\|S_{\lambda}f\|^2&=&2|\alpha_0|^2+4\sum\limits_{j=1}^{\infty}|\alpha_{1j}|^2 + 4|\alpha_{21}|^2+|\alpha_{22}|^2+16\sum\limits_{j=3}^{\infty}|\alpha_{2j}|^2 \\
			&\|S_{\lambda}^*S_{\lambda}f\|^2 &=& 4|\alpha_0|^2+16\sum\limits_{j=1}^{\infty}|\alpha_{1j}|^2 + 16|\alpha_{21}|^2+|\alpha_{22}|^2+256\sum\limits_{j=3}^{\infty}|\alpha_{2j}|^2 \\
			&\|S_{\lambda}^3f\|^2&=&20|\alpha_0|^2+64\sum\limits_{j=1}^{\infty}|\alpha_{1j}|^2 + 64|\alpha_{21}|^2+256|\alpha_{22}|^2+{64}^2\sum\limits_{j=3}^{\infty}|\alpha_{2j}|^2.
		\end{align*}
		Clearly, $\|S_{\lambda}^*S_{\lambda}f\|^2\leq \|S_{\lambda}^3f\|\|S_{\lambda}f\|$ for any choice of $f$. Hence $S_{\lambda}$ is quasi-$\ast$-paranormal.
	\end{example}
	It is known that for an absolutely norm attaining positive operator, the essential spectrum is a singleton set. Here we give an example of a positive operator on a directed tree whose essential spectrum is not a singleton set.
	\begin{example}
		Consider the graph $G=(V,E)$ as in Example~\ref{ex} with weights, 
		$\lambda_{ij}=i$. Clearly the operator $N_{\lambda}$ defined by $N_{\lambda}e_v= \lambda_ve_v$ with given weights is absolutely norm attaining on $\mathcal{B}(\ell^2(V))$ with $\|N_{\lambda}\|=2$. Also the essential spectrum, $\sigma_{ess}(N_{\lambda})=\{1,2\}$.
	\end{example}
	\begin{example}
		Consider the graph $G=(V,E)$ with a root vertex having 2 branches with weights $\lambda_{ij}=\begin{cases}
			i & \mbox{ if $j=1,2$}\\
			i-\frac{1}{j-1} & \mbox{ if $j>2$}
		\end{cases}$.
		Clearly the operator $N_{\lambda}$ defined as above is absolutely norm attaining on $\mathcal{B}(\ell^2(V))$ with $\|S_{\lambda}\|=2$. Also the essential spectrum, $\sigma_{ess}(N_{\lambda})=\{1,2\}$. Here 1 and 2 are limits of a decreasing sequence in the spectrum.
	\end{example}
	The above examples shows that the spectrum of a positive operator has atmost one limit point or an eigenvalue with infinite multiplicity is not true in every space. In graph setting, we observe that there can be more than one (possibly infinite) limit points/ eigenvalues with infinite multiplicity.
	\section{Weighted Composition Operators on Directed Graphs}
	In this section, we study $\ast$-paranormal and quasi-$\ast$-paranormal weighted composition operators on $L^2(V, \mathscr{V},\mu)$.
	
	 Let $u, w: V \to \mathbb{C}$ be two measurable functions. The Lambert operator on $L^2(\mu)$ is defined as $E_{u,w}f:= w\cdot E(uf)$ for every $f \in L^2(\mu)$. In \cite{lambert quasi}, Estaremi showed that $E_{u,w}$ is quasi-$\ast$-paranormal if and only if $E(|u|^2)E(|w|^2)\leq |E(uw)|^2$ a.e. on $S(E(|w|^2))$ where $S(f)$ denotes the support of $f$ defined by $S(f)=\{v\in V: f(v)\neq 0\}$. 
\begin{theorem}
Let $W:=u.f\circ \psi$ be a weighted composition operator on $L^2(\mu)$. Then
\begin{enumerate}
	\item $W$ is $\ast$-paranormal if and only if 
		\[h_2E_2(|u|_2^2)\circ \psi^{-2}f-2 \lambda u\cdot h_1\circ \psi \cdot E_1(\bar{u}f) + \lambda^2 f \geq 0 \text{ for every }f \in L^2(\mu),~ \lambda\geq 0.\]
		\item $W$ is quasi-$\ast$-paranormal if and only if 
		\[\left(h_3\cdot E_3(|u|_3^2)\circ \psi^{-3} - 2\lambda (h_1\cdot E_1(|u|^2)\circ \psi^{-1})^2+ \lambda^2 (h_1\cdot E_1(|u|^2)\circ\psi^{-1}) \right)f \geq 0\] for every $f \in L^2(\mu),~ \lambda\geq 0.$
\end{enumerate}
\end{theorem} 
	\begin{proof}
	We have $W^{*^n}W^n= h_n\cdot E_n(|u|_n^2)\circ \psi^{-n} $ \cite{m-isometry graph} such that  
		\[E_n(|u|_n^2)\circ \psi^{-n}(v)= \begin{cases}
			K^r_{i, j+p}& \text{ if } v=v^r_{i,j} \text{ for } r\in J_m, i\in J_{\eta_r}, \text{ and }j\in \mathbb{N},\\
			K^r_p& \text{ if } v=v_r \text{ for } r\in J_m.
		\end{cases}\]
		where $K^r_{i, j+n}= |u|_n^2(v^r_{i, j+n})$ 
		and 	
		$\displaystyle K_n^r= \frac{|u|_n^2(v_{{\Psi}_(n+r)})\mu(v_{{\Psi}_(n+r)})+\sum\limits_{j=1}^n \sum\limits_{\substack{s=1\\  \Psi(n+r)=\Psi(s+j)}}^m \sum\limits_{i=1}^{\eta_s}|u|_n^2(v^s_{i,j})\mu(v^s_{i,j})}{\mu(v_{{\Psi}_(n+r)})+\sum\limits_{j=1}^n \sum\limits_{\substack{s=1\\ \Psi_2(n+r)=\Psi(s+j)}}^m \sum\limits_{i=1}^{\eta_s}\mu(v^s_{i,j})}.$
	Also we can easily calculate $WW^*f= u\cdot h_1\circ \psi \cdot E_1(\bar{u}f)$. Since $W$ is $\ast$-paranormal if and only if $W^{*^2}W^2-2 \lambda WW^* + \lambda^2 \geq 0$ and quasi-$\ast$-paranormal if and only if \\$W^{*^3}W^3-2 \lambda (W^*W)^2 + \lambda^2 W^*W\geq 0$ for every $f \in L^2(\mu)$. By substitution, we have the results.
	\end{proof}
	\begin{example}
		Consider the directed graph with $m=3, \eta_1=2$ and $\eta_2=\eta_3=0$. Let \[u(v)=\begin{cases}
			1 & \text{ if } v= v_r, r\in J_3 \\
			2 & \text{ if } v=v_{i,j}^1,i\in J_2, j\in \mathbb{N}
		\end{cases}\] and $\mu(v_1)=\mu(v_2)=1, \mu(v_3)=1$ and $\mu(v_{i,j}^1)=\frac{1}{2}$ for $i\in J_2, j\in \mathbb{N}$. Then the weighted composition operator induced by $u$ and $\psi$ is quasi-$\ast$-paranormal operator. 
	\end{example}
	\begin{example}
		Consider the directed graph with $m=3, \eta_1=2$ and $\eta_2=\eta_3=0$. Let \[u(v)=\begin{cases}
			\frac{1}{4} & if v= v_r, r\in J_3 \text{ or } v= v_{i,1}^1 , i\in J_2 \\
			4 & if v=v_{i,j}^1,i\in J_2, j\in \mathbb{N}\setminus \{1\}
		\end{cases}.\] If we are using the  counting measure, the weighted composition operator induced by $u$ and $\psi$ is quasi-$\ast$-paranormal operator.  
	\end{example}
	Now we discuss the norm of weighted composition operator on $L^2(\mu)$.
		We have $W(f)= \sum\limits_{v \in V}u(v)\cdot (f\circ \psi)(v)$ and hence $\|Wf\|^2=  \sum\limits_{v \in V}|u(v)\cdot (f\circ \psi)(v)|^2$. Since $\|W\|^2=\|W^*W\|$, we have $\|W\|^2=\sup |h_1\cdot E(|u_1|^2)\circ \psi^{-1}(v)|=
	\sup |h_1(v)\cdot|u(v)|^2|.$
	\begin{example}
		Consider the directed graph on Example 4.2. We have
		\[|h_1(v)\cdot|u(v)|^2|=\begin{cases}
			2 & \text{ if } v= v_1, \\
			1 & \text{ if } v=v_1, v_2\\
		4 & \text{ if } v=v_{i,j}^1,i\in J_2, j\in \mathbb{N}
		\end{cases}.\] Hence $\|W\|=2=\|We_{i,j}^1\|$ where $i \in J_2,j>1$. So $W$ is norm attaining.
	\end{example}
	
		Let $(V,E)$ be a directed graph. Let $u:V \to \mathbb{C}$ be two measurable functions and $f \in \ell^2(V)$. Let $\psi: V \to V$ be a non-singular measurable function. The weighted composition operator, $W:=W_{u,\psi}: \ell^2(V) \to \ell^2(V)$ defined by,
	\[Wf(w)= u(w)\cdot f(\psi(w)).\]
	The adjoint operator is defined as 
	\[W^*f(w)= \sum\limits_{v \in chi(w)}\overline{u(v)}f(v).\]
	 Now we have the following theorems on $\ast$-paranormal and quasi-$\ast$-paranormal operators.
	\begin{theorem}
		Let $W$ be a weighted composition operator defined on $\ell^2(V)$, then the following are equivalent.
		\begin{enumerate}
			\item $W$ is $\ast$-paranormal
			\item $\sum\limits_{w\in V}\left|\sum\limits_{v\in chi(w)}\overline{u(v)}g(v)\right|^2 \leq \left(\sum\limits_{w\in V}|u_2(w)|^2|g\circ \psi^2(w)|^2\right)^{1/2}\left(\sum\limits_{w\in V}|g(w)|^2\right)^{1/2}$
			\item  $\sum\limits_{w\in V}\left(\sum\limits_{v \in \psi^{-2}(w)}|u_2(v)|^2-2k\cdot u(w)\sum\limits_{v_1\in \psi^{-1}(w)}\overline{u(v)}+k^2  \right)g(w) \geq 0$ for every $k \geq 0$ and $g \in \ell^2(V)$.
		\end{enumerate}
		
	\end{theorem}
	\begin{proof}
For (1)$\iff$ (3).\\We have $W$ is $\ast$-paranormal if and only if $\|W^*g\|^2\leq \|W^2g\|\|g\|$ for every $f \in \ell^2(V)$. By simple calculations, we have $\|W^*f\|^2=\sum\limits_{w\in V}\left|\sum\limits_{v\in chi(w)}\overline{u(v)}g(v)\right|^2$ and $\|W^2g\|^2=\sum\limits_{w\in V}|u_2(w)|^2|g\circ \psi^2(w)|^2$. The result follows by substitution.\\		
For (1)$\iff$ (3).\\We have $W$ is $\ast$-paranormal if and only if  $W^{*^2}W^2-2k WW^* +k^2 \geq 0 $ for every $f \in \ell^2(V)$ and $k >0$. For $n \in \mathbb{N}$, we can find
 \[W^{*^n}W^nf(w)=\sum\limits_{v \in \psi^{-n}(w)}|u_n(v)|^2f(w).\]
	So we have $W^{*^2}W^2f(w)=\sum\limits_{v \in \psi^{-2}(w)}|u_2(v)|^2f(w)$ and $WW^*f = u(w)\sum\limits_{v_1\in \psi^{-1}(w)}\overline{u(v)}f(w)$. Hence the required result follows.
	\end{proof}
	\begin{theorem}
		Let $W$ be a weighted composition operator defined on $\ell^2(V)$, then the following are equivalent.
		\begin{enumerate}
			\item $W$ is quasi-$\ast$-paranormal
			\item $\sum\limits_{w\in V}\left|\sum\limits_{v\in chi(w)}|u(v)|^2\right|^2|f(w)|^2 \leq \left(\sum\limits_{w\in V}|u_3(w)|^2|f\circ \psi^3(w)|^2\right)^{1/2}\left(\sum\limits_{w\in V}|u_1(w)|^2|f\circ \psi(w)|^2\right)^{1/2}$
			\item  $\sum\limits_{w\in V}\left(\sum\limits_{v \in \psi^{-3}(w)}|u_3(v)|^2-2k (\sum\limits_{v_1\in \psi^{-1}(w)}|u(v_1)|^2)^2+k^2 \sum\limits_{v_1\in \psi^{-1}(w)}|u(v_1)|^2 \right)f(w) \geq 0$ for every $k \geq 0$ and $f \in \ell^2(V)$.
		\end{enumerate}
	\end{theorem}
	\begin{proof}
		Since $W$ is quasi-$\ast$-paranormal, we have $W^{*^3}W^3-2k (W^*W)^2+k^2W^*W \geq 0$ and $\|W^*Wf\|^2\leq \|W^3f\|\|Wf\|$ for every $k \geq 0$ and $f \in \ell^2(V)$. Since, $W^{*^n}W^nf(w)=\sum\limits_{v \in \psi^{-n}(w)}|u_n(v)|^2f(w)$, theorem follows.
	\end{proof}
	
	Now we give an example of a quasi-$\ast$-paranormal Lambert operators.
	\begin{example}
	Consider the directed graph with $m=3, \eta_1=2$ and $\eta_2=\eta_3=0$.
		Let $u(v)=\begin{cases}
			2 & \text{ if }v=v_2\\
			\frac{1}{2} & \text{ if }v=v_{1,1}^1\\
			\frac{1}{2} & \text{ if }v=v_{2,1}^1
		\end{cases}$, $w(v)=\begin{cases}
			\frac{1}{2} & \text{ if }v=v_2\\
			\frac{1}{4} & \text{ if }v=v_{1,1}^1\\
			\frac{1}{4} & \text{ if }v=v_{2,1}^1
		\end{cases}$  ,
		$\mu(v)=\begin{cases}
			2 & \text{ if }v=v_2\\
			4 & \text{ if }v=v_{1,1}^1\\
			4 & \text{ if }v=v_{2,1}^1
		\end{cases}$ and all other $u(v), w(v)$ and $\mu(v)$ are positive. Then $E_{u,w}$ defined on $L^2(V, \mathscr{V}, \mu)$ is quasi-$\ast$-paranormal.
	\end{example}

	\section*{Acknowledgements}
	The first author acknowledges the University of Kerala, Kerala, India for the financial support through a University
	JRF for the period of 2025-2026, under the reference number 2460/2025/UOK
	dated 19.03.2025.
	
	\bibliographystyle{amsplain}

\end{document}